\documentclass[12pt]{amsart}
\usepackage{latexsym, amsmath, amssymb, amsthm, mathrsfs,wasysym}
\usepackage[alwaysadjust]{paralist}
\usepackage{caption}
\usepackage{subcaption}
\usepackage[T1]{fontenc}
\usepackage{lmodern}
\usepackage[backref,colorlinks=true,linkcolor=blue,urlcolor=blue, citecolor=blue]%
  {hyperref}
\usepackage[shortalphabetic]{amsrefs}
\usepackage[all]{xy}
\usepackage[T1]{fontenc}
\usepackage{mathptmx}
\usepackage{microtype}
\usepackage{framed}
\usepackage{tikz}
\usepackage{graphicx}
\usepackage[alwaysadjust]{paralist}
\makeatletter
\providecommand\@enum@widestlabel{7}
\makeatother

\usepackage[centering, includeheadfoot, hmargin=1in, vmargin=1in,
  headheight=30.4pt]{geometry}
\newtheorem{lemma}{Lemma}[section]
\newtheorem{theorem}[lemma]{Theorem}
\newtheorem{corollary}[lemma]{Corollary}
\newtheorem{proposition}[lemma]{Proposition}
\newtheorem{conjecture}[lemma]{Conjecture}

\newtheorem{claim}[lemma]{Claim}

\newtheorem*{claim*}{Claim}

\theoremstyle{definition}
\newtheorem{remark}[lemma]{Remark}
\newtheorem{example}[lemma]{Example}
\newtheorem{definition}[lemma]{Definition}

\renewcommand{\theequation}%
{\arabic{section}.\arabic{lemma}.\arabic{equation}}

\newcommand{\CC}{\ensuremath{\mathbb{C}}} 
\newcommand{\NN}{\ensuremath{\mathbb{N}}} 
\newcommand{\PP}{\ensuremath{\mathbb{P}}} 
\newcommand{\QQ}{\ensuremath{\mathbb{Q}}} 
\newcommand{\RR}{\ensuremath{\mathbb{R}}} 
\newcommand{\ZZ}{\ensuremath{\mathbb{Z}}} 
 
\newcommand{\sI}{\ensuremath{\kern -1pt \mathscr{I}\kern -2pt}} 
\newcommand{\sJ}{\ensuremath{\kern -2pt \mathscr{J}\kern -2pt}} 

\newcommand{\sO}{\ensuremath{\mathscr{O}}}

\renewcommand{\geq}{\geqslant}
\renewcommand{\leq}{\leqslant}

\DeclareMathOperator{\multlt}{mult}

\DeclareMathOperator{\Supp}{Supp}

\DeclareMathOperator{\mld}{mld}
\DeclareMathOperator{\lct}{lct}
\DeclareMathOperator{\wt}{wt}
\DeclareMathOperator{\modo}{mod}
\DeclareMathOperator{\slope}{slope}
\DeclareMathOperator{\mult}{mult}
\DeclareMathOperator{\rank}{rank}
\DeclareMathOperator{\centerr}{center}

\input xy
\xyoption{all}

\begin{document}

\title[]{Optimal bound for singularities on Fano type fibrations of relative dimension one}

\author{Bingyi Chen}
\address{Bingyi Chen, Department of Mathematics,
Sun Yat-sen University,
Guangzhou, 510275, P. R. China.}
\email{chenby253@mail.sysu.edu.cn, chenby16@tsinghua.org.cn}

\begin{abstract}
Let $\pi:X\rightarrow Z$ be a Fano type fibration with $\dim X-\dim Z=d$ and let $(X,B)$ be an $\epsilon$-lc pair with $K_X+B\sim_{\RR} 0/Z$.  The canonical bundle formula gives $(Z,B_Z+M_Z)$ where $B_Z$ is the discriminant divisor and $M_Z$ is the moduli divisor which is determined up to $\RR$-linear equivalence. Shokurov conjectured that one can choose $M_Z\geq 0$ such that $(Z,B_Z+M_Z)$ is $\delta$-lc where $\delta$ only depends on $d,\epsilon$. Very recently, this conjecture was proved by Birkar \cite{Bir23}. For $d=1$ and $\epsilon=1$, Han, Jiang and Luo \cite{HJL22} gave  the optimal value of $\delta=1/2$. In this paper, we give the optimal value of $\delta$ for $d=1$ and arbitrary $0<\epsilon\leq 1$.
\end{abstract}

\maketitle

\section{Introduction}
Throughout this paper we work over the field  of complex numbers $\CC$. We denote by $\NN$ the set of positive integers. 

In the process of  running minimal model program (MMP, for short), it is important to control the singularities that occur on the variety produced in each step. There are three types of steps in MMP: divisorial contractions, filps, Mori fiber spaces. In the first two cases, the singularities on the resulting variety are as good as those on the original variety. However, in the last case, for a Mori fiber space $f:X\rightarrow Z$ it is much more complicated to understand the singularities on $Z$. M\textsuperscript{c}Kernan conjectured that in this case the singularities on $Z$ are bounded in terms of those on $X$.

\begin{conjecture}[M\textsuperscript{c}Kernan]\label{conj:mckernan}
Fix a positive integer $d$ and a real number $\epsilon>0$. There exists $\delta>0$ depending only on $d,\epsilon$ and satisfying the following. Assume  

(1)~ $X$ is an $\epsilon$-lc variety of dimension $d$,

(2)~ $f:X\rightarrow Z$ is a Fano fibration (i.e. a contraction with $-K_X$ being ample over $Z$), and

(3) $K_Z$ is $\QQ$-Cartier.\\
Then $Z$ is $\delta$-lc.
\end{conjecture}
For a Fano fibration $X\to Z$ where $X$ is a threefold with only terminal singularities, Mori and Prokhorov \cite{MP08} proved that $Z$ is $1$-lc. 
The toric case of M\textsuperscript{c}Kernan's conjecture was confirmed by Alexeev and Borisov \cite{AB14}.

Shokurov proposed a more general conjecture which implies M\textsuperscript{c}Kernan's conjecture. Before stating this conjecture, we recall some background. Let $\pi:(X,B)\rightarrow Z$ be an lc-trivial fibration (see Definition \ref{def:canonical}). For example, $\pi:X\rightarrow Z$ is a contraction between normal varieties and   $(X,B)$ is an lc pair with $K_X+B\sim_{\RR} 0/Z$. By constructions of Kawamata \cites{Kaw97,Kaw98} and Ambro \cite{Amb05}, we have the so-called canonical bundle formula
$$K_X+B\sim_{\RR} \pi^*(K_Z+B_Z+M_Z),$$
where $B_Z$ is the discriminant part and $M_Z$ is the moduli part (see Definition \ref{def:KKA}). The moduli part $M_Z$ is determined up to $\RR$-linear equivalence and thus can be represented by different $\RR$-divisors. A projective variety $F$ is said to be of Fano type if it admits a klt pair $(F,\Delta)$ such that $-(K_F+\Delta)$ is ample. We are now ready to state Shokurov's conjecture.
\begin{conjecture}[Shokurov, cf. {\cite[Conjecture 1.5]{HJL22}}]\label{conj:shokurov}
Fix a positive integer $d$ and a real number $\epsilon>0$. There exists $\delta>0$ depending only on $d,\epsilon$ and satisfying the following. Let $\pi:(X,B)\rightarrow Z$ be an lc-trivial fibration and $z\in Z$ a codimension $\geq 1$ point such that

(1)~ $\dim X - \dim Z=d$,

(2)~ $\mld(X/Z\ni z,B)\geq \epsilon$, and

(3)~ the generic fiber of $\pi$ is of Fano type.\\
Then we can choose $M_Z\geq 0$ representing the moduli part such that $\mld(Z\ni z, B_Z+M_Z)\geq \delta$.
\end{conjecture}
Here we denote by $\mld(X/Z\ni z,B)$ (resp. $\mld(Z\ni z,B_Z+M_Z)$) the infimum of the log discrepancy of $E$ with respect to $(X,B)$ (resp. $(Z,B_Z+M_Z)$) where $E$ runs over all prime divisors over $X$ (resp. $Z$) whose image on $Z$ is $\overline{z}$.

Shokurov's conjecture contains more information about singularities on Fano type fibrations than M\textsuperscript{c}Kernan's conjecture. One of its consequences is that given a Fano fibration $X\to Z$ where $X$ is $\epsilon$-lc, the multiplicities of fibers over codimension 1 points on $Z$ are bounded from above in terms of $\epsilon$ and the relative dimension. As a related result, Mori and Prokhorov \cite{MP09} showed that if $X\rightarrow Z$ is a del Pezzo fibration from a terminal threefold onto a smooth curve, then the multiplicity of an irreducible fiber is at most 6.

\begin{remark}\label{remnew}
(a) With a slightly stronger condition (2') that $(X,B)$ is an $\epsilon$-lc pair, Birkar \cite{Bir16} proved Conjecture \ref{conj:shokurov} when either $\dim X-\dim Z=1$ or $(F,\Supp B|_F)$ is log bounded where $F$ is a general fiber of $\pi$. Combining the main results in \cites{Bir19,Bir21}, the conjecture under condition (2') holds when the coefficients of the horizontal part of $B$ are bounded from below by a fixed positive number.

(b) Birkar and Chen \cite{BC21} proved a modified version of Conjecture \ref{conj:shokurov} under condition (2') in the toric setting, which implies the boundedness of multiplicities of fibers over codimension 1 points for toric Fano fibrations.


(c) Very recently, Conjecture \ref{conj:shokurov} was proved by Birkar \cite{Bir23} under condition (2'). 
\end{remark}

In results mentioned in Remark \ref{remnew}, the existence of $\delta$ was proved but the explicit value of $\delta$ was not given. When $\dim X-\dim Z=1$ and $\epsilon\geq 1$, Han, Jiang and Luo \cite{HJL22} gave the optimal value of $\delta=\epsilon-1/2$ in Conjecture \ref{conj:shokurov}.

In this paper, we give the optimal value of $\delta$ in Conjecture \ref{conj:shokurov} for the case that $0<\epsilon\leq 1$ and $\dim X-\dim Z=1$.


\begin{theorem}\label{thm:main}
Let $\pi:(X,B)\rightarrow Z$ be an lc-trivial fibration and $z\in Z$ a codimension $\geq 1$ point such that

(1)~ $\dim X - \dim Z=1$,

(2)~ $\mld(X/Z\ni z,B)\geq \epsilon$ where $0<\epsilon\leq 1$, and

(3)~ the generic fiber of $\pi$ is a rational curve.\\
Then we can choose $M_Z\geq 0$ representing the moduli part such that $\mld(Z\ni z, B_Z+M_Z) \geq \delta,$ where
\begin{align*}
\delta=\frac{(2n+1)\epsilon-1}{2n(n+1)},\quad \text{when $\epsilon\in \left[\frac{1}{n+1},\frac{1}{n}\right]$ for some $n\in \NN$}.
\end{align*}
In particular, $\delta\geq \epsilon^2/2$ and the equality holds if and only if $\epsilon=1/n$ for some $n\in \NN$.
\end{theorem}
\begin{remark}

(a) The value of $\delta$ in Theorem \ref{thm:main} is optimal by Example \ref{exa}.

(b) As the generic fiber of $\pi$ is a rational curve, the condition that $\pi:(X,B)\rightarrow Z$ is an lc-trivial fibration is equivalent to the following condition: $\pi:X\rightarrow Z$ is a contraction between normal varieties with an lc sub-pair $(X,B)$ such that $K_X+B\sim_{\RR} 0/Z$ and $B$ is effective over generic point of $Z$ (see Remark \ref{remark:effective}).


\end{remark}

\begin{example}\label{exa}
Let $\epsilon=\frac{p}{q}$ be a rational number such that $\frac{1}{n+1}\leq \epsilon \leq \frac{1}{n}$, where $p,q,n\in \NN$.
Let $Z=\mathbb A^1$ with coordinate $t$ and let $X=Z \times \mathbb P^1$ with homogeneous coordinates $x,y$. Let $\pi:X\rightarrow Z$ be the projection and let $z$ be the origin of $Z$. Let $B$ be the $\QQ$-divisor
$$\frac{1}{qn(n+1)}\cdot \left(t^{q(2n+1)-p}x^{2qn(n+1)}+y^{2qn(n+1)}=0\right)+\frac{qn-p}{q(n+1)}\cdot (t=0).$$
Then $K_X+B\sim_{\mathbb Q} 0/Z$. By Lemma \ref{lemexa} we have
$$\mld(X/Z\ni z,B)=\epsilon \quad  \text{ and }\quad \lct(X/Z\ni z,B; \pi^*z)= \frac{(2n+1)\epsilon-1}{2n(n+1)}.$$
So the coefficient of $z$ in the discriminant divisor $B_Z$ is equal to $$1-\frac{(2n+1)\epsilon-1}{2n(n+1)}.$$
\end{example}

As a corollary, we have the following global version of Theorem \ref{thm:main} which involves less technical notations.

\begin{corollary}\label{newcor}
Let $(X,B)$ be a pair and $\pi:X\rightarrow Z$ a contraction between normal varieties such that

(1) $\dim X-\dim Z=1$,

(2) $a(E,X,B)\geq \epsilon$ for any prime divisor $E$ over $X$ with $\pi(\centerr_X(E))\neq Z$, where $0<\epsilon\leq 1$,

(3) $K_X+B\sim_{\mathbb R} 0/Z$,

(4) the generic fiber of $\pi$ is a rational curve.\\
Then we can choose $M_Z\geq 0$ representing the moduli part such that $(Z,B_Z+M_Z)$ is $\delta$-lc, where $\delta$ is as in Theorem \ref{thm:main}.
\end{corollary}

Applying Corollary \ref{newcor}, we give an effective value of $\delta$ in a more general form of Conjecture \ref{conj:mckernan} when the relative dimension is 1.
\begin{corollary}\label{thm:main2}
Let $\pi:X\rightarrow Z$ be a contraction between normal varieties such that 

(1) $\dim X - \dim Z=1$,

(2) $X$ is $\epsilon$-lc, where $0<\epsilon\leq 1$,

(3) $-K_X$ is big and nef over $Z$.\\
Then 

(1) $Z$ is $\delta$-lc if $K_Z$ is $\QQ$-Cartier, and 

(2) for any codimension 1 point $z\in Z$, the multiplicity of each component of $\pi^*z$ is bounded from above by $1/\delta$, where $\delta$ is as in Theorem \ref{thm:main}.
\end{corollary}


Following the idea in \cite{HJL22}, we may reduce Theorem \ref{thm:main} to the case that $\pi:X\rightarrow Z$ is a $\PP^1$-bundle over a smooth curve, and finally reduce to a local problem on estimating the log canonical threshold of a smooth curve with respect to a pair on a smooth surface. Then we solve this local problem by proving the following theorem.

\begin{theorem}\label{thm:surface}
Let $(X\ni P,B)$ be a germ of surface pair such that $X$ is smooth and $\mld(X\ni P,B)\geq \epsilon$, where $0<\epsilon\leq 1$. Let $C$ be a smooth curve on $X$ passing through $P$ such that $\mult_C B \leq 1-\epsilon$ and $(B'\cdot C)_P\leq 2$, where $B'=B-\mult_C B \cdot C$. Then $$\lct(X\ni P,B;C)\geq \delta,$$
where $\delta$ is as in Theorem \ref{thm:main}.
\end{theorem}

The idea of the proof of Theorem \ref{thm:surface} is as follows. Let $\Gamma_+(B+tC)$ be the Newton polytope of $B+tC$ (see Definition \ref{def:newton}), which depends on the choice of local coordinates of $X\ni P$. One can show that 
\begin{align}\label{pp}
\text{if $(X,B+tC)$ is lc near $P$, then $\textbf{1}=(1,1)\in \Gamma_+(B+tC)$}
\end{align}
(see Lemma \ref{lemma:1 in Newton}). However, in general the converse of \eqref{pp} does not hold. For example, let $X=\mathbb A^2$ with coordinates $x,y$ and $B=(x^2+2xy+y^2=0)$, then $\textbf{1}\in \Gamma_+(B)$ while $(X,B)$ is not lc. The key point is that, by making use of \cite[Theorem 1]{Kaw17} we can pick suitable local coordinates such that the converse of \eqref{pp} holds. Then
$$\lct(X\ni P,B;C)=\sup\{t>0\mid \textbf{1}\in \Gamma_+(B+tC)\}$$
and hence we can estimate its lower bound by combinatorics.

\medskip

\noindent\textbf{Acknowledgments.} The author would like to thank Caucher Birkar for his valuable comments and constant support. This work was partially done at Tsinghua University. It was completed at Sun Yat-sen University with support of the start-up fund from Sun Yat-sen University.


\section{Preliminaries}
We will freely use the standard notations and definitions in \cites{KM98,BCHM10}. 
\subsection{Sub-pairs and singularities}
A contraction $\pi:X\rightarrow Z$ is a projective morphism of varieties with $\pi_*\mathcal{O}_X=\mathcal{O}_Z$. 

\begin{definition}[cf. {\cite[Definition 2.4]{HJL22}}]
A sub-pair $(X,B)$ consists of a normal variety $X$ and an $\RR$-divisor $B$ on $X$ such that $K_X+B$ is $\RR$-Cartier. A sub-pair $(X,B)$ is called a pair if $B$ is effective.

A (relative) sub-pair $(X/Z\ni z,B)$ consists of  a contraction $\pi:X\rightarrow Z$ between normal varieties, a (not necessarily closed) point $z\in Z$ and an $\RR$-divisor $B$ on $X$ such that $K_X+B$ is $\RR$-Cartier.
A sub-pair $(X/Z\ni z,B)$ is called a (relative) pair if $B$ is effective. In the case that $Z=X$, $z=x$ and $\pi$ is the identity morphism, we will use  $(X\ni x,B)$ instead of $(X/Z\ni z,B)$. A sub-pair $(X\ni x,B)$ is called a germ if $x$ is a closed point.
\end{definition}

\begin{definition} [cf. {\cite[Definition 2.5]{HJL22}}]
Let $(X,B)$ be a sub-pair and $E$ a prime divisor over $X$, i.e. a prime divisor on a normal variety $Y$ with a birational morphism $g:Y\rightarrow X$. We may write $K_Y+B_Y:=g^*(K_X+B)$. Then the center of $E$ is defined as the image of $E$ on $X$ under the morphism $g$ and the log discrepancy $a(E,X,B)$ of $E$ with respect to $(X,B)$ is defined as $1-\multlt_E B_Y$.

Let $(X/Z\ni z,B)$ be a sub-pair with contraction $\pi:X\rightarrow Z$. The minimal log discrepancy of $(X/Z\ni z,B)$ is defined as
$$\mld(X/Z\ni z,B):=\inf\{a(E,X,B) \mid \text{$E$ is a prime divisor over $X$ with $\pi(\centerr_X(E))=\overline{z}$}\}.$$
\end{definition}

\begin{definition}[cf. {\cite[Definition 2.7]{HJL22}}]
A sub-pair $(X/Z\ni z,B)$ is said to be $\epsilon$-lc (resp. $\epsilon$-klt, lc, klt) if $\mld(X/Z\ni z,B)\geq \epsilon$ (resp. $>\epsilon$,~$\geq 0$,~$> 0$). A sub-pair $(X,B)$ is said to be $\epsilon$-lc (resp. $\epsilon$-klt, lc, klt) if $(X\ni x,B)$ is so for any codimension $\geq 1$ point $x\in X$. In the case when $B=0$, we also say $X$ is $\epsilon$-lc (resp. $\epsilon$-klt, lc, klt). We remark that $(X/Z\ni z,B)$ is lc if and only if $(X,B)$ is lc over some neighborhood of $z\in Z$ (see \cite[Lemma 2.8]{HJL22}).
\end{definition}  

\begin{definition} [cf. {\cite[Definition 2.10]{HJL22}}]
Let $(X/Z\ni z,B)$ be a sub-pair with contraction $\pi:X\rightarrow Z$ and let $D$ be an effective $\RR$-Cartier $\RR$-divisor on $X$. The log canonical threshold of  $D$ with respect to  $(X/Z\ni z,B)$ is defined as 
$$\lct(X/Z\ni z,B;D):=\sup\{t\geq 0\mid (X/Z\ni z,B+tD) \text{ is lc}\}.$$
If $z\in Z$ is a codimension 1 point, then $\overline{z}$ is a Cartier divisor on some neighborhood $U$ of $z\in Z$. We define
$$\lct(X/Z\ni z,B;\pi^*\overline{z}):=\sup\{t\geq 0\mid (X/Z\ni z,B+t\pi^*\overline{z}) \text{ is lc over $U$}\}.$$
This definition is independent of the choice of $U$.
\end{definition}


\subsection{b-divisors}
Let $X$ be a normal variety. A b-$\RR$-divisor $\textbf{D}$ (or b-divisor for short) is a collection of $\RR$-divisors $\textbf{D}_Y$ for each birational model $Y$ over $X$, such that for any birational morphism $\sigma:Y_1\rightarrow Y_2/X$ we have $\sigma_*\textbf{D}_{Y_1}=\textbf{D}_{Y_2}$. 

We say that a b-divisor $\textbf{D}$ is b-semiample if there is a birational model $Y_0$ over $X$ such that $\textbf{D}_{Y_0}$ is a semiample $\RR$-Cartier $\RR$-divisor and for any birational morphism $\sigma:Y\rightarrow Y_0/X$ we have $\textbf{D}_{Y}=\sigma^* \textbf{D}_{Y_0}$. 

\subsection{Canonical bundle formula}
\begin{definition}[{\cite[Definition 3.2]{FG14}}]\label{def:canonical}
An lc-trivial fibration $\pi:(X,B)\rightarrow Z$ consists of a contraction $\pi:X\rightarrow Z$ between normal varieties and a sub-pair $(X,B)$ such that

(1)~ $(X/Z\ni \eta,B)$ is lc, where $\eta$ is the generic point of $Z$;

(2)~ $\rank \pi_* \sO_X(\lceil -B_{<1} \rceil)=1$, where $B_{<1}:=\sum_{b_i<1} b_i B_i$ for $B=\sum b_i B_i$ with prime divisors $B_i$'s;

(3)~ $K_X+B \sim_{\RR} 0/Z$, i.e. there is an $\RR$-Cartier $\RR$-divisor $L$ on $Z$ such that $K_X+B\sim_{\RR} \pi^*L$.

\end{definition}
\begin{remark}[cf. {\cite[Remark 2.13]{HJL22}}]\label{remark:effective}
If the generic fiber of $\pi$ is a rational curve, then the second condition is equivalent to that $B$ is effective over the generic point of $Z$.
\end{remark}

\begin{definition}[\cites{Kaw97,Kaw98}, \cite{Amb05}]\label{def:KKA}
Let $\pi:(X,B)\rightarrow Z$ be an lc-trivial fibration. Then  $K_X+B\sim_{\RR} \pi^* L$ for some $\RR$-Cartier $\RR$-divisor $L$ on $Z$. For any prime divisor $D$ on $Z$, denote
$$b_D:=\lct(X/Z\ni \eta_D, B; \pi^*\overline{\eta_D})$$
where $\eta_D$ is the generic point of $D$. We set $B_Z=\sum_D (1-b_D)D$ where $D$ runs over all prime divisors on $Z$. We set $M_Z=L-K_Z-B_Z$, which is determined up to $\RR$-linear equivalence since $L$ is so. Then we have the so-called canonical bundle formula
$$K_X+B\sim_{\RR} \pi^*(K_Z+B_Z+M_Z).$$
The part $B_Z$ is called the discriminant part and the part $M_Z$ is called the moduli part. We remark that $B_Z$ is effective if $B$ is so.

Let $\sigma: Z'\rightarrow Z$ be a birational morphism from a normal variety $Z'$ and let $X'$ be the resolution of the main component of $X\times_Z Z'$ with induced morphism $\tau:X'\rightarrow X$ and $\pi':X'\rightarrow Z'$. Write $K_{X'}+B'=\tau^*(K_X+B)$, then $K_{X'}+B'\sim_{\RR} \pi'^* \sigma^* L$. Similarly we can define the discriminant part $B_{Z'}$ and the moduli part $M_{Z'}$ for the contraction $(X',B')\rightarrow Z'$. One can check that $\sigma_* B_{Z'}=B_Z$ and $\sigma_* M_{Z'}=M_Z$. Hence there exist b-divisors $\textbf{B}$ and $\textbf{M}$ such that $\textbf{B}_{Z'}=B_{Z'}$ and $\textbf{M}_{Z'}=M_{Z'}$ for any birational model $Z'$ over $Z$, which are called the discriminant b-divisor and the moduli b-divisor respectively.
\end{definition}

\begin{theorem}[{\cite[Theorem 8.1]{PS09}}, cf. {\cite[Theorem 2.14]{HJL22}}] Let $\pi:(X,B)\rightarrow Z$ be an lc-trivial fibration such that 
$B$ is a $\QQ$-divisor, $\dim X-\dim Z=1$ and the generic fiber of $\pi$ is a rational curve. Then the moduli b-divisor ${\bf M}$ is b-semiample.
\end{theorem}

\section{Log canonical thresholds on smooth surfaces}
In this section, we use Newton diagrams to study the local problem on estimating the lower bound of the log canonical threshold of a smooth curve with respect to a pair on a smooth surface germ, which will play a key role in the proof of Theorem \ref{thm:main}.
\subsection{Newton diagrams}
In this subsection we collect some basic definitions and facts about Newton diagrams.
\begin{definition}\label{def:newton}
Let $f\in \CC[[x,y]]$ be a nonzero power series in two variables. We may write
$$f=\sum_{(p,q)\in \ZZ_{\geq 0}} c_{p,q} x^p y^q.$$
Define the Newton polytope $\Gamma_+(f)$ of $f$ as the convex hull in $\RR^2$ of 
$$\bigcup_{c_{p,q}\neq 0} (p,q)+\RR_{\geq 0}^2$$
and define the Newton diagram $\Gamma(f)$ of $f$ as the boundary of $\Gamma_+(f)$. Then $\Gamma(f)$ is the union of 1-dimensional faces of $\Gamma_+(f)$ (including two non-compact faces).

Let $X\ni P$ be a germ of smooth surface with local coordinates $(x,y)$ and let $B$ be an effective $\RR$-divisor on $X$. Write $B=\sum_i b_i B_i$ where each $b_i$ is a positive real number and each $B_i$ is a prime divisor. The completion of $B_i$ along $P$ is defined by $(f_i=0)$ for some $f_i\in\CC[[x,y]]$. We define the Newton polytope $\Gamma_+(B)$ of $B$ as $\sum_i b_i\cdot \Gamma_+(f_i)$ and define the Newton diagram $\Gamma(B)$ as the boundary of $\Gamma_+(B)$. Note that the definition of $\Gamma_+(B)$ and $\Gamma(B)$ do not depend on the choice of $f_i$ up to a unit in $\CC[[x,y]]$, but depend on the choice of local coordinates.
\end{definition}

\begin{definition}\label{def:leftrightvertex}
Let $B$ be an effective $\RR$-divisor on a germ of smooth surface $X\ni P$ with local coordinates $(x,y)$ and let $S$ be a 1-dimensional face of $\Gamma(B)$. If $S$ is a compact face, then it has a normal vector having positive coordinates. Write its two vertices as $(p_1,q_1)$ and $(p_2,q_2)$ with $p_1<p_2$ and $q_1>q_2$. We call $(p_1,q_1)$ the left vertex of $S$ and call $(p_2,q_2)$ the right vertex of $S$. If $S$ is a non-compact face, then $S$ is parallel to either $x$-axis or $y$-axis. So we may write $S$ as either $(p_1+\RR_{\geq 0},q_1)$ or $(p_2,q_2+\RR_{\geq 0})$. In the former (resp. latter) case we call $(p_1,q_1)$ (resp. $(0,+\infty)$) the left vertex of $S$ and call $(+\infty,0)$ (resp. $(p_2,q_2)$) the right vertex of $S$. 
\end{definition}


Next we recall some facts about weighted blow-ups. Let $X\ni P$ be a germ of smooth surface with local coordinates $(x,y)$ and let $\textbf{w}=(w_1,w_2)$ with two coprime positive integers $w_1,w_2$. Denote by $E_{\textbf{w}}$ the exceptional divisor obtained by the weighted blow up $\sigma:Y\rightarrow X$ at $P\in X$ with $\wt(x,y)=(w_1,w_2)$. Then 
$$a(E_{\textbf{w}},X,0)=w_1+w_2=\langle  \textbf{w},\textbf{1}\rangle,$$
where we write $\textbf{1}$ for the vector $(1,1)$ and write $\langle, \rangle$ for the inner product of $\RR^2$. Let $B$ be an effective $\RR$-divisor on $X$. Then 
$$\mult_{E_{\textbf{w}}} \sigma^*B=\langle \textbf{w}, \Gamma(B) \rangle:=\inf\{\langle \textbf{w}, \textbf{p} \rangle\mid \textbf{p} \in \Gamma(B)\}.$$
Therefore 
\begin{equation}\label{eqNewton:surface}
a(E_{\textbf{w}},X,B)=\langle \textbf{w},\textbf{1}\rangle-\langle \textbf{w},\Gamma(B)\rangle.
\end{equation}

\begin{lemma}\label{lemma:1 in Newton}
With the above notations, if $\mld(X\ni P, B)\geq 0$, then ${\bf 1} \in \Gamma_+(B)$.
\end{lemma}
\begin{proof}
If $\textbf{1}\notin \Gamma_+(B)$, there exists a  1-dimensional face of $\Gamma(B)$ with normal vector $\textbf{e}=(e_1,e_2)$ such that $e_1+e_2< \langle \textbf{e},\Gamma(B) \rangle$. 
After some small perturbations, we may suppose that $e_1,e_2$ are positive rational numbers and the above inequality still holds. Multiplying with some positive rational number, we may suppose that $e_1,e_2$ are coprime positive integers. Let $E_{\textbf{e}}$ be the exceptional divisor obtained by the weighted blow up at $P\in X$ with $\wt(x,y)=(e_1,e_2)$. Then
$a(E_{\textbf{e}},X,B)<0$,  which contradicts our hypothesis.
\end{proof}

\begin{lemma}\label{lemma:translation}
Let $B,D$ be two effective $\RR$-divisors on a germ of smooth surface $X \ni P$ with local coordinates $(x,y)$. Assume $\Gamma(D)$ has only two vertices $(p,0)$ and $(0,q)$, where $p,q$ are positive numbers. Let $S$ be a 1-dimensional face of $\Gamma(B+tD)$ with left vertex $(p_1,q_1)$ and right vertex $(p_2,q_2)$, where $p_1,q_1,p_2,q_2$ are non-negative real numbers or $+\infty$. 
\begin{enumerate}[(1)]
\item If $(q_1-q_2)/(p_2-p_1)>q/p$, then $S-(0,tq)$ is a face of $\Gamma(B)$.
\item If $(q_1-q_2)/(p_2-p_1)<q/p$, then $S-(tp,0)$ is a face of $\Gamma(B)$.
\item If $(q_1-q_2)/(p_2-p_1)=q/p$, denote $(p'_1,q'_1)=(p_1,q_1-tq)$ and $(p'_2,q'_2)=(p_2-tp,q_2)$, then the line segment connecting $(p'_1,q'_1)$ and $(p'_2,q'_2)$ is a face of $\Gamma(B)$ (or a vertex of $\Gamma(B)$ when these two points coincide). Moreover we have 
$$\langle{\bf v},\Gamma(B)\rangle=p'_1q+q'_1p=p'_2q+q'_2p,$$
where ${\bf v}=(q,p)$.
\end{enumerate}
\end{lemma}
\begin{proof}
The Newton diagram $\Gamma(B+tD)$ is obtained by gluing the 1-dimensional faces of $\Gamma(B)$ and $t\cdot\Gamma(D)$, suitably displaced such that the slopes of faces are more and more gentle form left to right (cf. \cite[Chapter II, Remark 2.17.1]{GLS07}). See the figure below. The lemma follows directly from this fact.
\end{proof}

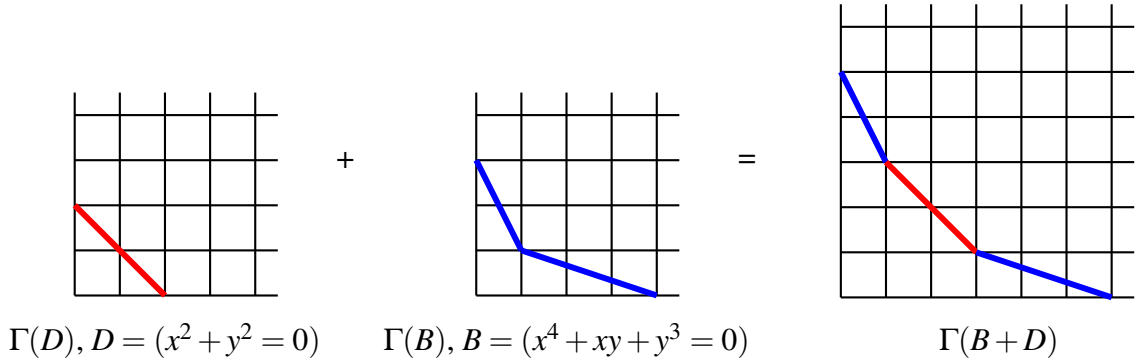
\begin{figure}[ht]
\begin{tikzpicture}[scale=0.6] 
\draw[thick,-] (0,0) -- (4.5,0);
\draw[thick,-] (0,1) -- (4.5,1);
\draw[thick,-] (0,2) -- (4.5,2);
\draw[thick,-] (0,3) -- (4.5,3);
\draw[thick,-] (0,4) -- (4.5,4);
\draw[thick,-] (0,0) -- (0,4.5);
\draw[thick,-] (1,0) -- (1,4.5);
\draw[thick,-] (2,0) -- (2,4.5);
\draw[thick,-] (3,0) -- (3,4.5);
\draw[thick,-] (4,0) -- (4,4.5);
\draw[thick,-][line width=2.25][color=red] (0,2)--(2,0);
\node[] at (6,3) {+};
\node[] at (2,-1) {$\Gamma(D)$, $D=(x^2+y^2=0)$};
\end{tikzpicture}
\begin{tikzpicture}[scale=0.6] 
\draw[thick,-] (0,0) -- (4.5,0);
\draw[thick,-] (0,1) -- (4.5,1);
\draw[thick,-] (0,2) -- (4.5,2);
\draw[thick,-] (0,3) -- (4.5,3);
\draw[thick,-] (0,4) -- (4.5,4);
\draw[thick,-] (0,0) -- (0,4.5);
\draw[thick,-] (1,0) -- (1,4.5);
\draw[thick,-] (2,0) -- (2,4.5);
\draw[thick,-] (3,0) -- (3,4.5);
\draw[thick,-] (4,0) -- (4,4.5);
\draw[thick,-][line width=2.25][color=blue] (0,3)--(1,1);
\draw[thick,-][line width=2.25][color=blue] (1,1)--(4,0);
\node[] at (6,3) {=};
\node[] at (2,-1) {$\Gamma(B)$, $B=(x^4+xy+y^3=0)$};
\end{tikzpicture}
\qquad
\begin{tikzpicture}[scale=0.6] 
\draw[thick,-] (0,0) -- (6.5,0);
\draw[thick,-] (0,1) -- (6.5,1);
\draw[thick,-] (0,2) -- (6.5,2);
\draw[thick,-] (0,3) -- (6.5,3);
\draw[thick,-] (0,4) -- (6.5,4);
\draw[thick,-] (0,5) -- (6.5,5);
\draw[thick,-] (0,6) -- (6.5,6);

\draw[thick,-] (0,0) -- (0,6.5);
\draw[thick,-] (1,0) -- (1,6.5);
\draw[thick,-] (2,0) -- (2,6.5);
\draw[thick,-] (3,0) -- (3,6.5);
\draw[thick,-] (4,0) -- (4,6.5);
\draw[thick,-] (5,0) -- (5,6.5);
\draw[thick,-] (6,0) -- (6,6.5);

\draw[thick,-][line width=2.25][color=blue] (0,5)--(1,3);
\draw[thick,-][line width=2.25][color=red] (1,3)--(3,1);
\draw[thick,-][line width=2.25][color=blue] (3,1)--(6,0);
\node[] at (3.5,-1) {$\Gamma(B+D)$};
\end{tikzpicture}

\end{figure}

\begin{lemma}\label{lemma:t2}
Let $B,D$ be two effective $\RR$-divisors on a germ of smooth surface $X \ni P$ with local coordinates $(x,y)$. Assume $\Gamma(D)$ has only one vertex $(p,0)$ (resp.$(0,q)$), where $p$ (resp. $q$) is a positive number. Let $S$ be a 1-dimensional face of $\Gamma(B+tD)$. Then $S-(tp,0)$ (resp. $S-(0,tq)$) is a face of $\Gamma(B)$.
\end{lemma}
\begin{proof}
The lemma follows directly from the fact that $\Gamma_+(B+tD)=\Gamma_+(B)+t\cdot\Gamma_+(D)$.
\end{proof}

Let $X\ni P$ be a germ of smooth surface with local coordinates $(x,y)$ and let $C$ be a smooth curve on $X$ passing $P$. The completion of $C$ along $P$ is defined by $(g=0)$ for some $g\in \CC [[x,y]]$. Since $C$ is smooth, after possibly switching $x$ and $y$, we may suppose that the term $x$ appears in the expression of $g$ with non-zero coefficient. Therefore $\Gamma(C)$ has either one vertex $(1,0)$ or two vertices $(1,0)$ and $(0,b)$, where $b$ is a positive integer. In the former case, we also say $\Gamma(C)$ has two vertices $(1,0)$ and $(0,+\infty)$.

\begin{lemma}\label{lemma:intersection}
Let $B$ be an effective $\RR$-divisor on a germ of smooth surface $X \ni P$ with coordinates $(x,y)$ and let $C$ be a smooth curve passing $P$. Suppose that $C \nsubseteq \Supp B$ and $\Gamma(C)$ has two vertices $(1,0)$ and $(0,b)$, where $b$ is a positive integer or $+\infty$. Then 
$$(B\cdot C)_P\geq \langle {\bf b}, \Gamma(B)  \rangle$$
where ${\bf b}=(b,1)$. Here we use the convention that $+\infty\cdot 0=0$.
\end{lemma}
\begin{proof}
We may suppose that $B$ is a prime divisor with its completion defined by $(f=0)$ for some $f\in\CC[[x,y]]$. Indeed, for the general case $B=\sum_i b_iB_i$, since $(B\cdot C)_P=\sum_i b_i(B_i\cdot C)_P$ and $\Gamma_+(B)=\sum_i b_i\cdot \Gamma_+(B_i)$, it suffices to show the lemma for each $B_i$. 

If $b$ is a positive integer, by \cite[Chapter I, Theorem 3.3]{GLS07} there exists $y(t)\in t\cdot \CC[[t]]$ such that $\varphi:t\mapsto (t^b,y(t))$ is a parametrization of $C$. As $C$ is smooth, the multiplicity of $y(t)$ is 1. For any term $x^py^q$ appears in the expression of $f$ with non-zero coefficient, the multiplicity of $\varphi^*(x^py^q)=t^{pb}\cdot y(t)^q$ is $pb+q\geq \langle \textbf{b},\Gamma(B)\rangle$. Therefore $(B\cdot C)_P\geq \langle \textbf{b} ,\Gamma(B)\rangle$.

If $b=+\infty$, then $\varphi:t\mapsto (0,t)$ is a parametrization of $C$.  Since $C \nsubseteq \Supp B$,  there exists a non-negative integer $I$ such that $y^{I}$ appears in the expression of $f$ with non-zero coefficient. 
Then 
$$(B\cdot C)_P= \min\{I\mid (0,I)\in \Gamma_+(B)\}= \langle \textbf{b},\Gamma(B)\rangle,$$
where $\textbf{b}=(+\infty,1)$.
\end{proof}

\subsection{Proof of Theorem \ref{thm:surface}}
In this subsection we give the proof of Theorem \ref{thm:surface}.
\begin{lemma} \label{lemma:mld=0}
Let $(X\ni x,B)$ be an lc pair such that $x$ is a codimension 2 point. Let $D$ be a prime divisor on $X$. Denote $t=\lct(X\ni x,B;D)$. If $\mult_D B+t<1$, then $\mld(X\ni x, B+tD)=0$.
\end{lemma}
\begin{proof}
As $(X\ni x, B+tD)$ is lc, after shrinking $X$ near $x$, we may suppose that $(X,B+tD)$ is an lc pair. Let $f:Y\rightarrow X$ be a log resolution of $(X,B+D)$. We may write 
$$K_Y+\sum_i (1-a(E_i,X,B+tD))E_i+B_Y+tD_Y=f^*(K_X+B+tD)$$
where $B_Y$, $D_Y$ are the strict transforms of $B,D$ on $Y$ respectively and $E_i$ runs over all exceptional prime divisors on $Y$ over $X$. Shrinking $X$ near $x$, we may suppose that $\overline{x} \subseteq \centerr_X(E_i)$ for each $E_i$, which implies that $\centerr_X(E_i)=\overline{x}$ since $x$ is a codimension 2 point and $E_i$ is exceptional.

We claim that $a(E_i,X,B+tD)=0$ for some $E_i$. Indeed, if this is not the case, then $a(E_i,X,B+tD)>0$ for all $E_i$. Take $t'$ such that $0<t'-t\ll 1$. Then $a(E_i,X,B+t'D)>0$ for all $E_i$ and coefficients in $B_Y+t'D_Y$ are at most 1 since $\mult_D B+t<1$. As $B_Y+D_Y+\sum_i E_i$ is simple normal crossing, it follows that 
$$(Y,\sum_i (1-a(E_i,X,B+t'D))E_i+B_Y+t'D_Y)$$ 
is lc, which is the crepand pull back of $(X,B+t'D)$. Hence $\lct(X\ni P,B;D)\geq t'>t$, which leads to a contradiction.

By the above claim, there is $E_i$ such that $a(E_i,X,B+tD)=0$. Note that $(X,B+tD)$ is lc and  $\centerr_X(E_i)=\overline{x}$, we have $\mld(X\ni x, B+tD)=0$.
\end{proof}

\begin{proof}[Proof of Theorem \ref{thm:surface}]
Suppose that $\frac{1}{n+1}\leq \epsilon\leq \frac{1}{n}$ for some $n\in \NN$. Let $t=\lct(X\ni P,B;C)$. We need to show that 
$$t\geq \frac{(2n+1)\epsilon-1}{2n(n+1)}.$$
If $t\geq \epsilon$, there is nothing further to prove. So we may suppose that $t<\epsilon$. Then $\mld(X\ni P, B+tC)=0$ by Lemma \ref{lemma:mld=0}. By \cite[Theorem 1]{Kaw17}, there are local coordinates $x,y$ and two coprime positive integers $w_1,w_2$ such that 
\begin{equation}\label{eq4:surface}
a(E_{\textbf{w}},X,B+tC)=\mld(X\ni P, B+tC)=0,
\end{equation}
where $E_{\textbf{w}}$ is the exceptional divisor obtained by the weighted blow up at $P\in X$ with $\wt(x,y)=(w_1,w_2)$. For any $\mathbb R$-divisor $D$ on $X$, we denote by $\Gamma(D)$ its Newton diagram with respect to local coordinates $(x,y)$ (see Definition \ref{def:newton}).

From \eqref{eqNewton:surface} and \eqref{eq4:surface} we deduce that
\begin{equation}\label{eq3:surface}
\langle \textbf{w},\textbf{1}\rangle=\langle \textbf{w},\Gamma(B+tC)\rangle,
\end{equation}
where $\textbf{w}=(w_1,w_2)$. On the other hand, $\textbf{1}\in \Gamma_+(B+tC)$ by Lemma \ref{lemma:1 in Newton}. This implies that $$\textbf{1}\in\Gamma(B+tC).$$
\begin{claim} \label{claim:vertex}
If ${\bf 1}$ lies in a non-compact 1-dimensional face $F$ of $\Gamma(B+tC)$, then ${\bf 1}$ is a vertex of $F$.
\end{claim}
\begin{proof}
Without loss of generality, we may suppose that $F$ is parallel to $y$-axis. We may write $F=(1,s+\RR_{\geq 0})$ since $\textbf{1}\in F$. If $\textbf{1}$ is not the vertex of $F$, then $0\leq s<1$. It follows that $w_1+sw_2<w_1+w_2$, which contradicts \eqref{eq3:surface}. 
\end{proof}

Take $S$ to be the unique 1-dimensional face of $\Gamma(B+tC)$ such that $\textbf{1}\in S$ and $\textbf{1}$ is not the right vertex of $S$ (see Definition \ref{def:leftrightvertex} for the definition of right vertex). Then $S$ is not parallel to $y$-axis by the above claim. We denote the left (resp. right) vertex of $S$ by $(p_1,q_1)$ (resp. $(p_2,q_2)$). The slope of $S$ is defined to be $$\slope(S):=(q_1-q_2)/(p_2-p_1).$$
Then $\slope(S)$ is a non-negative real number (since $S$ is not parallel to $y$-axis, the slope $<+\infty$). 

Since $C$ is a smooth curve passing $P$, after possibly switching $x$ and $y$, we may suppose that $\Gamma(C)$ has two vertices $(1,0)$ and $(0,b)$, where $b$ is a positive integer or $+\infty$ (see the argument before Lemma \ref{lemma:intersection}). Denote $a=\mult_C B$. Then $a\leq 1-\epsilon$. We divide the rest of the proof in three cases based on the slope of $S$.

\fbox{\textit{Case} 1:} $\slope(S)=b$. Then $0<b=\slope(S)<+\infty$. 
It follows from $\textbf{1} \in S$ that 
\begin{equation}\label{eq1:surface}
b+1=bp_2+q_2.
\end{equation}
Let $\textbf{b}=(b,1)$ and  $E_{\textbf{b}}$  the exceptional divisor obtained by the weighted blow up $\sigma:Y\rightarrow X$ at $P\in X$ with $\wt(x,y)=(b,1)$. Then $\mult_{E_{\textbf{b}}} \sigma^* C=\langle \textbf{b},\Gamma(C) \rangle=b$ and
$$a(E_{\textbf{b}},X,B+tC)=\langle \textbf{b},\textbf{1}\rangle-\langle \textbf{b},\Gamma(B+tC) \rangle= b+1-(bp_2+q_2)=0.$$ It follows that $a(E_{\textbf{b}},X,B)=tb$. As $\mld(X\ni P, B)\geq \epsilon$, we have $t\geq \epsilon/b$.

From Lemma \ref{lemma:translation}(3), Lemma \ref{lemma:intersection} and the condition that $(B-aC\cdot C)_P\leq 2$, we deduce that 
\begin{equation}\label{eq2:surface}
b(p_2-t-a)+q_2= \langle \textbf{b}, \Gamma(B-aC)\rangle\leq 2.
\end{equation}
Combining \eqref{eq1:surface} and  \eqref{eq2:surface}, we get
$$t\geq 1-a-\frac{1}{b}\geq \epsilon-\frac{1}{b}.$$
Therefore $t\geq \max\{\epsilon/b,\epsilon-1/b\}.$

If $b\geq 2n+1$, then 
$$t\geq\epsilon-\frac{1}{b}\geq\epsilon-\frac{1}{2n+1}=\frac{(2n+1)\epsilon-1}{2n+1} \geq \frac{(2n+1)\epsilon-1}{2n(n+1)}.$$
Otherwise  $b\leq 2n$, then
$$t\geq \frac{\epsilon}{b}\geq \frac{\epsilon}{2n}=\frac{(n+1)\epsilon}{2n(n+1)}\geq \frac{(n+1)\epsilon+(n\epsilon-1)}{2n(n+1)}=\frac{(2n+1)\epsilon-1}{2n(n+1)}$$
where the third inequality follows from $ \epsilon\leq 1/n$.

\fbox{\textit{Case} 2:} $\slope(S)=0$. By Claim \ref{claim:vertex} we may write $S=(1+\RR_{\geq 0},1)$. Then $(1-t+\RR_{\geq 0},1)$ is a face of $\Gamma(B)$ by Lemma \ref{lemma:translation}(2) when $b<+\infty$ and by Lemma \ref{lemma:t2} when $b=+\infty$. Denote $\textbf{s}=(1-t,1)$. Let $\textbf{p}=(1,N)$ and  $E_{\textbf{p}}$  the exceptional divisor obtained by the weighted blow up at $P\in X$ with $\wt(x,y)=(1,N)$, where $N$ is a sufficiently large integer. Then 
\begin{align*}
\epsilon\leq a(E_{\textbf{p}},X,B) &=\langle \textbf{p},\textbf{1} \rangle-\langle \textbf{p},\Gamma(B) \rangle=\langle \textbf{p},\textbf{1} \rangle-\langle \textbf{p},\textbf{s}\rangle\\
&=(1+N)-((1-t)+N)=t.
\end{align*}
So we have
$$t\geq\epsilon\geq \frac{(2n+1)\epsilon-1}{2n(n+1)}.$$

\fbox{\textit{Case} 3:} $0<\slope(S)<b$. Denote $S'=S-(t+a,0)$. Then $S'$ is a compact 1-dimensional face of $\Gamma(B-aC)$ by Lemma \ref{lemma:translation}(2) when $b<+\infty$ and by Lemma \ref{lemma:t2} when $b=+\infty$. Denote the intersection point of the line through $S'$ with $x$-axis (resp. $y$-axis) by $(\alpha,0)$ (resp. $(0,\beta)$), where $\alpha$ and $\beta$ are positive real numbers. Then $\beta/\alpha< b$. As $\textbf{1}\in S=S'+(t+a,0)$, one has 
\begin{equation}\label{eq:lalala}
\beta\geq 1\quad \text{ and }\quad  t=1+\alpha/\beta-\alpha-a.
\end{equation}
Note that $\Gamma_+(B-aC)$ is contained in the convex hull of the union of $(\alpha,0)+\RR_{\geq 0}^2$ and $(0,\beta)+\RR_{\geq 0}^2$. So we have 
\begin{equation}\label{eq:lalala2}
\beta=\min\{\beta,b\alpha\}\leq \langle \textbf{b},\Gamma(B-aC)\rangle\leq 2,
\end{equation}
where $\textbf{b}=(b,1)$ and the last inequality follows from Lemma \ref{lemma:intersection} and  $(B-aC\cdot C)_P\leq 2$.

For any $\textbf{e}=(e_1,e_2)\in \NN^2$ with $e_1,e_2$ coprime, we have 
$$a(E_{\textbf{e}},X,B-aC)=e_1+e_2-\langle \textbf{e}, \Gamma(B-aC) \rangle\leq e_1+e_2-\min\{\alpha e_1,\beta e_2\},$$
where $E_{\textbf{e}}$ is the exceptional divisor obtained by the weighted blow up $\sigma:Y\rightarrow X$ at $P\in X$ with $\wt(x,y)=(e_1,e_2)$.
Since $\mult_{E_{\textbf{e}}} \sigma^* C=\min\{e_1,be_2\}$ and $\beta/\alpha<b$, one has
\begin{equation}\label{eq:long}
\epsilon\leq a(E_{\textbf{e}},X,B)
\leq \begin{cases}
e_1+e_2-\alpha e_1- ae_1, & \text{if $e_1/e_2 \leq \beta/\alpha$,}\\
e_1+e_2-\beta e_2- ae_1,  & \text{if $\beta/\alpha\leq e_1/e_2\leq b$,} \\
e_1+e_2-\beta e_2- abe_2, & \text{if $e_1/e_2\geq b$.}
\end{cases}
\end{equation}
Define a piece-wise linear function $f$ on $\RR_{\geq 0}^2$ by
$$f(r_1,r_2) :=
\begin{cases}
r_1+r_2-\alpha r_1- ar_1, & \text{if $r_1/r_2 \leq \beta/\alpha$,}\\
r_1+r_2-\beta r_2- ar_1,  & \text{if $r_1/r_2\geq \beta/\alpha$,} \\
\end{cases}
$$ 
for any $(r_1,r_2)\in \RR_{\geq 0}^2$. Here we use the convention that $r_1/0=+\infty\geq \beta/\alpha$. By \eqref{eq:long}  one has $f(e_1,e_2)\geq \epsilon$ for any $(e_1,e_2)\in \NN^2$ such that $e_1/e_2\leq b$. In particular
$f(b,1)\geq \epsilon$ if $b<+\infty$. 
\begin{claim*}
$f(e_1,e_2)\geq \epsilon$ for any $e_1\in \NN$ and $e_2\in \ZZ_{\geq 0}$.
\end{claim*}
\begin{proof}
If $e_2=0$, then $f(e_1,e_2)=(1-a)e_1\geq \epsilon e_1\geq \epsilon$. So we may suppose that $e_2>0$.
Since the claim has been confirmed when $e_1/e_2\leq b$, we may suppose that $b<+\infty$ and $e_1>be_2$. Then $f(e_1,e_2)=(1-a)e_1+(1-\beta)e_2$ since $b> \beta/\alpha$.

If $e_1<b$, then $e_2<e_1/b<1$, which implies that $e_2\leq 0$ and contradicts our hypothesis.
So we may assume that $e_1\geq b$. Then 
$$f(e_1,e_2)\geq f(e_1,e_1/b)=f(b,1)\cdot e_1/b\geq \epsilon,$$
where the first inequality follows from $1-\beta\leq 0$ (see \eqref{eq:lalala}) and $e_2<e_1/b$.
\end{proof}

By \eqref{eq:lalala} we have $t=f(1,\alpha/\beta)$. By Lemma \ref{lemma:mod} below, there exist $k\in \NN$ and $l\in \mathbb Z_{\geq 0}$ such that 
$$\frac{\epsilon-|k\alpha/\beta-l| }{k}\geq \frac{(2n+1)\epsilon-1}{2n(n+1)}.$$

We claim that
$$f(k,k\alpha/\beta)\geq f(k,l)-|k\alpha/\beta-l|.$$
Indeed, if $l\geq k\alpha/\beta$, then 
$$f(k,l)=k+l-\alpha k-ak=f(k,k\alpha/\beta)+(l-k\alpha/\beta);$$
Otherwise $l\leq k\alpha/\beta$, then 
\begin{align*}
f(k,l)=k+l-\beta l -ak=& f(k,k\alpha/\beta)+(\beta-1)(k\alpha/\beta-l)\\
\leq & f(k,k\alpha/\beta)+(k\alpha/\beta-l),
\end{align*}
where the inequality follows from $1\leq \beta \leq 2$ (see \eqref{eq:lalala} and \eqref{eq:lalala2}).

Therefore we have 
$$t=f(1,\alpha/\beta)=\frac{f(k,k\alpha/\beta)}{k}\geq \frac{f(k,l)-|k\alpha/\beta-l| }{k}\geq \frac{\epsilon-|k\alpha/\beta-l| }{k}\geq \frac{(2n+1)\epsilon-1}{2n(n+1)}.$$

\fbox{\textit{Case} 4:} $\slope(S)>b$. Then $0<b<\slope(S)<+\infty$. Denote $S'=S-(0,b(t+a))$. By Lemma \ref{lemma:translation}(1) $S'$ is a compact 1-dimensional face of $\Gamma(B-aC)$. Denote the intersection point of the line through $S'$ with $x$-axis (resp. $y$-axis) by $(\alpha,0)$ (resp. $(0,\beta)$), where $\alpha$ and $\beta$ are positive real numbers. Then $\beta> b\alpha$ and  $\Gamma_+(B-aC)$ is contained in the convex hull of the union of $(\alpha,0)+\RR_{\geq 0}^2$ and $(0,\beta)+\RR_{\geq 0}^2$. Recall that $\textbf{1}\in S$ and $\textbf{1}$ in not the right vertex of $S$, the $x$-coordinate of the right vertex of $S'$ is large than 1, which implies that $\alpha>1$. 
Since $(B-aC\cdot C)_P\leq 2$, by Lemma \ref{lemma:intersection} one has 
$$b\alpha=\min\{\beta,b\alpha\}\leq \langle\textbf{b},\Gamma(B-aC)\rangle\leq 2,$$
where $\textbf{b}=(b,1)$. Therefore $b=1$.
After switching $x$ and $y$, we have $0<\slope(S)<1=b$ and hence we can reduce this to Case 3 (note that the condition that $\textbf{1}$ in not the right vertex of $S$ may not be satisfied after switching $x$ and $y$; fortunately this condition is not used in the proof for Case 3). This finishes the proof of Theorem \ref{thm:surface}.
\end{proof}

\begin{lemma}\label{lemma:mod}
Let $x,\epsilon$ be two real numbers so that $x\geq 0$ and $\frac{1}{n+1}\leq \epsilon\leq \frac{1}{n}$ for some $n\in \NN$. Then there exist $k\in \NN$ and $l\in \mathbb Z_{\geq 0}$ such that 
$$\frac{\epsilon-|kx-l| }{k}\geq \frac{(2n+1)\epsilon-1}{2n(n+1)}.$$
\end{lemma}
\begin{proof}
Without loss of generality, we may suppose that $0\leq x<1$. If $x=0$, we can take $k=1$ and $l=0$, then we are done. So we may suppose that $0<x<1$. 

We construct two finite sequences as follows. Set $r_{-1}=1$, $r_0=x$, $a_{-1}=0$ and $a_{0}=1$. For $i\geq 1$, assuming that $r_{i-2}, r_{i-1}, a_{i-2},a_{i-1}$ have been defined, we define $r_i$  by $r_{i-2}=b_{i}r_{i-1}+r_{i}$ where $b_i$ is  a positive integer and $0\leq r_{i}<r_{i-1}$
and we define $a_i=a_{i-2}+b_i a_{i-1}$. If either $r_i=0$ or $a_i\geq n+1$, we terminate the sequence and set $m$ to be $i$. Note that the sequence will stop because $a_i> a_{i-1}$ for $i\geq 2$. Moreover, by definition we have $m\geq 1$ and $a_i$ is a postive integer for $0\leq i\leq m$.
\begin{claim*} For $i=-1,0,\cdots,m$, we have 
$$a_i x \equiv
\begin{cases}
r_i ~(\modo 1),  & \text{if $i$ is even,} \\
-r_i~(\modo 1), & \text{if $i$ is odd.}
\end{cases}
$$
\end{claim*}
\begin{proof}[Proof of the claim]
By definition we have $r_{-1}=1-a_{-1} x$ and $r_0=a_0 x$. Suppose that the claim holds for $i-2$ and $i-1$. 
If $i$ is even, then
\begin{align*}
r_i &=r_{i-2}-b_i r_{i-1}\\
&\equiv a_{i-2}x-b_i(-a_{i-1}x) ~(\modo 1)\\
& =(a_{i-2}+b_ia_{i-1})x=a_i x.
\end{align*}
If $i$ is odd, then
\begin{align*}
r_i&=r_{i-2}-b_i r_{i-1}\\
&\equiv -a_{i-2}x-b_i a_{i-1}x ~(\modo 1)\\
&=-(a_{i-2}+b_ia_{i-1})x=-a_i x.
\end{align*}
\end{proof}
\begin{claim*} $a_ir_{i-1}+a_{i-1}r_i=1$ for $i=0,\cdots,m$.
\end{claim*}
\begin{proof}[Proof of the claim]
By definition we have $a_{0}r_{-1}+a_{-1}r_0=1$. Suppose that the claim holds for $i-1$. Then
\begin{align*}
&a_ir_{i-1}+a_{i-1}r_i\\
=&(a_{i-2}+b_ia_{i-1})r_{i-1}+a_{i-1}(r_{i-2}-b_ir_{i-1})\\
=&a_{i-2}r_{i-1}+a_{i-1}r_{i-2}=1.
\end{align*}
\end{proof}

If $a_m\leq n$, by definition we have $r_m=0$, so $a_mx$ is an integer. Taking $k=a_m$ and $l=a_mx$, we have 
$$\frac{\epsilon-|kx-l| }{k}=\frac{\epsilon-0}{a_m}\geq  \frac{\epsilon}{n}>\frac{(2n+1)\epsilon-1}{2n(n+1)}.$$  

If $a_m\geq n+1$, by definition we have $a_{m-1}\leq n$. We claim that
\begin{align}\label{eqI}
  \left(\frac{1}{a_{m-1}}+\frac{1}{a_m}\right)\epsilon-\frac{1}{a_{m-1}a_{m}}\geq \left(\frac{1}{n}+\frac{1}{n+1}\right)\epsilon-\frac{1}{n(n+1)}.
\end{align}
Indeed, the left side minus the right side of \eqref{eqI} is equal to
\begin{align*}
&\left(\frac{1}{a_{m-1}}-\frac{1}{n}\right)\epsilon-\frac{1}{a_ma_{m-1}}-\left(\frac{1}{n+1}-\frac{1}{a_m}\right)\epsilon+\frac{1}{n(n+1)}\\
=& \left(\frac{1}{a_{m-1}}-\frac{1}{n}\right)\epsilon-\frac{1}{a_ma_{m-1}}+\frac{1}{na_m}-\left(\frac{1}{n+1}-\frac{1}{a_m}\right)\epsilon+\frac{1}{n(n+1)}-\frac{1}{na_m}\\
=& \left(\frac{1}{a_{m-1}}-\frac{1}{n}\right)\left(\epsilon-\frac{1}{a_m}\right)+\left(\frac{1}{n+1}-\frac{1}{a_m}\right)\left(\frac{1}{n}-\epsilon\right)\geq 0.
\end{align*}
By \eqref{eqI} and the fact that $a_mr_{m-1}+a_{m-1}r_m=1$, we have
$$ \left(\frac{1}{a_{m-1}}+\frac{1}{a_m}\right)\epsilon-\frac{a_mr_{m-1}+a_{m-1}r_m}{a_{m-1}a_{m}}\geq \left(\frac{1}{n}+\frac{1}{n+1}\right)\epsilon-\frac{1}{n(n+1)},$$
which is equivalent to
$$\frac{\epsilon-r_{m-1}}{a_{m-1}}+\frac{\epsilon-r_{m}}{a_{m}}\geq \frac{(2n+1)\epsilon-1}{n(n+1)}. $$
So we have
$$\text{either }\quad  \frac{\epsilon-r_{m-1}}{a_{m-1}}\geq \frac{(2n+1)\epsilon-1}{2n(n+1)}\quad  \text{ or } \quad \frac{\epsilon-r_{m}}{a_{m}}\geq \frac{(2n+1)\epsilon-1}{2n(n+1)}.$$
In the former case we take $k=a_{m-1}$ and in the latter case we take $k=a_m$. This completes the proof of Lemma \ref{lemma:mod}.
\end{proof}

\section{Proofs of main theorems}
Making a small modification to the proof of \cite[Theorem 1.10]{HJL22} and applying Theorem \ref{thm:surface}, we obtain the proof of the following proposition.
\begin{proposition}\label{prop:codim1}
Let $\pi:(X,B)\rightarrow Z$ be an lc-trivial fibration and $z\in Z$ a codimension $1$ point such that

(1) $\dim X - \dim Z=1$,

(2) $\mld(X/Z\ni z,B)\geq \epsilon$ where $0<\epsilon\leq 1$, and

(3) the generic fiber of $\pi$ is a rational curve.\\
Then $\lct(X/Z\ni z,B; \pi^*\overline{z})\geq \delta$, where $\delta$ is as in Theorem \ref{thm:main}.
In particular, if $B$ is effective, then the multiplicity of each component of $\pi^*z$ is bounded from above by $1/\delta$.
\end{proposition}
\begin{proof}
If  $\mld(X/Z\ni z,B)> 1$,  by \cite[Theorem 1.10]{HJL22}, $$\lct(X/Z\ni z,B; \pi^*\overline{z})\geq \mld(X/Z\ni z,B)-\frac{1}{2}> \frac{1}{2}.$$ 
It is easy to check that $\delta\leq \frac{1}{2}$ where $\delta$ is as in Theorem \ref{thm:main} when $0<\epsilon\leq 1$. So we may assume that $\mld(X/Z\ni z,B)\leq 1$.

By the same argument as that in the proof of \cite[Theorem 1.10]{HJL22}, we can reduce the problem to the case when $\dim X=2$. So from now on we suppose that $X$ is a surface and $Z$ is a curve.

First we consider the case when $X$ is smooth and $B\geq 0$. Since the generic fiber of $\pi$ is a rational curve, by running an MMP$/Z$ for $K_X$ we can assume that $\pi: X\rightarrow Z$ is a $\PP^1$-bundle over a curve. Then $F:=\pi^*(z)\cong \PP^1$. Applying the adjunction formula we obtain $K_X\cdot F=-2$. As $(K_X+B)\cdot F=0$ and $F\cdot F=0$, one has $(B'\cdot F)_P\leq 2$ for any close point $P\in F$ where $B'=B-\mult_F B\cdot F$. 
By Theorem \ref{thm:surface}, $\lct(X\ni P, B;F)\geq \delta$ for any closed point $P\in F$, where $\delta$ is as in Theorem \ref{thm:main}. It follows that $\lct(X/Z\ni z,B; \pi^*z)\geq \delta$.

Next we treat the general case. Let $f:W\rightarrow X$ be a log resolution of $(X,B+\pi^*z)$. We may write $K_W+B_W=f^*(K_X+B)$. Then $\mld(W/Z\ni z,B_W)=\mld(X/Z\ni z,B)$. So we have 
\begin{align}\label{eqnew}
\epsilon\leq \mld(W/Z\ni z,B_W) \leq 1,
\end{align} which implies that $(W/Z\ni z,B_W)$ is lc. Hence $B_W\leq 1$ (i.e. coefficients of $B_W$ are at most 1) after possibly shrinking $Z$ near $z$. We claim that 
\begin{align}\label{text}
\text{there is a prime divisor $C_0\subseteq \Supp f^*\pi^*z$ such that $\mult_{C_0} B_W\geq 0$.}
\end{align}
Indeed, if this is not the case, since $B_W\leq 1$ and $B_W$ is a simple normal crossing divisor, 
one has $\mld(W/Z\ni z,B_W)> 1$, which is in contradiction with \eqref{eqnew}.  

We may write $B_W=D-G$ where $D,G\geq 0$ have no common components. Then 
$$K_W+D=K_W+B_W+G\sim_{\RR} G/Z.$$
We claim that any $(K_W+D)$-MMP/$Z$ is also a $K_W$-MMP/$Z$. Indeed, if $E$ is a curve on $W$ such that its image on $Z$ is a point and $E\cdot (K_W+D)<0$, then $E\cdot G<0$ and hence $E\subseteq \Supp G$. So $E\nsubseteq \Supp D$ and $E\cdot D\geq 0$. It follows that $E\cdot K_W <0$.

By Remark \ref{remark:effective}, $B$ is effective over the generic point of $Z$. So we may assume that $\Supp G\subseteq \Supp f^*\pi^* z$ after possibly shrinking $Z$ near $z$. By \eqref{text}, there exists a prime divisor $C_0 \subseteq \Supp f^*\pi^* z$ but $C_0 \nsubseteq \Supp G$. Hence $G$ is very exceptional over $Z$ (see \cite[Definition 3.1]{Bir12}). By \cite[Theorem 3.5]{Bir12}, we may run a $(K_W+D)$-MMP/$Z$ and reach a model $Y$ such that $G_Y=0$ with the induced maps $g:W\rightarrow Y$ and $h:Y\rightarrow Z$. As this MMP is also a $K_W$-MMP/Z, $Y$ is a smooth surface. Since $K_W+B_W=K_W+D-G\sim_{\RR} 0/Z$, by the negativity lemma, one has $g^*(K_Y+D_Y)=K_W+B_W$. Hence 
$$\mld(Y/Z\ni z,D_Y)=\mld(X/Z\ni z,B)=\epsilon$$  and  $$\lct(Y/Z\ni z,D_Y;h^*z)=\lct(X/Z\ni z,B;\pi^*z).$$ Replacing $(X,B)$ by $(Y,D_Y)$, we can reduce the general case to the case when $X$ is smooth and $B$ is effective, which has been solved.
\end{proof}

\begin{proof}[Proof of Theorem \ref{thm:main}]
By the same argument as that in the proof of \cite[Theorem 1.7]{HJL22},  Theorem \ref{thm:main} follows from Proposition \ref{prop:codim1}.
\end{proof}

\begin{proof}[Proof of Corollary \ref{newcor}]
By the same argument as that in the proof of \cite[Corollary 1.8]{HJL22},  Corollary \ref{newcor} follows from Proposition \ref{prop:codim1}.
\end{proof}

\begin{proof}[Proof of Corollary \ref{thm:main2}]
The case when $\epsilon=1$ is solved in \cite[Theorem 1.4]{HJL22}. So we may suppose that $\epsilon<1$. By replacing $X$ by its anti-canonical model over $Z$, we may suppose that $-K_X$ is ample over $Z$. Take a large integer $N$ such that $1/N\leq 1-\epsilon$ and $-NK_X$ is very ample over $Z$. Let $H$ be a general effective divisor such that $H\sim -NK_X/Z$ and let $B=\frac{1}{N} H$. Then $K_X+B\sim_{\QQ} 0/Z$ and $(X,B)$ is an $\epsilon$-lc pair. By Corollary \ref{newcor}, we can choose $M_Z\geq 0$ representing the moduli part so that $(Z,B_Z+M_Z)$ is $\delta$-lc, where $\delta$ is as in Theorem \ref{thm:main}. 
Since  $B\geq 0$, one has $B_Z\geq 0$. Hence $Z$ is $\delta$-lc. The last statement of Corollary \ref{thm:main2} follows from the last statement of Proposition \ref{prop:codim1}.
\end{proof}
\appendix
\section{Proof of the lemma in Example \ref{exa}}
In this appendix we will prove the lemma used in Example \ref{exa}.
\begin{lemma}\label{mld}
Let $X=\mathbb A^2$ with coordinates $x,y$. Let $B$ be the $\mathbb R$-divisor 
$$\lambda\cdot (x^m+y^n=0)+\mu\cdot (x=0)$$ 
where $n,m\in \mathbb N$ and $0\leq \lambda, \mu \leq 1$, such that $(X,B)$ is lc. 
Then
$$\mld(X\ni 0, B)=\inf\big\{p_1+p_2-\min\{(\lambda m+\mu)p_1,\mu p_1 +\lambda n p_2\} \mid (p_1,p_2)\in \mathbb N^2 \big\}.$$
\end{lemma}
\begin{proof}
Denote $B_1=\lambda\cdot (x^m+y^n=0)$ and $B_2=\mu\cdot (x=0)$. Then $B=B_1+B_2$. 
For any $\textbf{p}=(p_1,p_2)\in \mathbb N^2$ with $p_1,p_2$ coprime, denote by $E_{\textbf{p}}$ the toric divisor over $X$ corresponding to the ray $\mathbb R_{\geq 0}\cdot \textbf{p}$. Then
$$a(E_{\textbf{p}},X,B)=p_1+p_2-\min\{(\lambda m+\mu)p_1,\mu p_1 +\lambda n p_2\}.$$ So it suffices to show that there exists a toric divisor $F$ over $X$ with  $\text{center}_{X} F=0$ such that $a(F,X,B)=\mld(X\ni 0, B)$. 

Let $g=\gcd(n,m)$ and $(n',m')=(n,m)/g$. Let $\sigma: Y\rightarrow X$ be the weighted blow up at $0$ with $\wt(x,y)=(n',m')$. Then $Y\subset X\times \mathbb P^1_{z,w}$ is defined by $(x^{m'}w=y^{n'}z)$ and the exceptional divisor $E$ on $Y$ is isomorphic to $\mathbb P^1_{z,w}$. Denote $P_1=[0:1]$ and $P_2=[1:0]$ on $E$. Then $Y\setminus \{P_1,P_2\}$ is smooth. Denote by $B'_1,B'_2$ the birational transform of $B_1,B_2$ on $Y$. Then we can write
$$K_Y+e  E + B'_1+B'_2=\sigma^*(K_X+B).$$
Moreover, $(\Supp B'_1)|_E$ is defined by $(z^g+w^g=0)$ on $E$. Hence $P_1,P_2\notin E \cap \Supp B'_1$ and $E$ meets $\Supp B'_1$ transversally. It follows that
$$\mld(X\ni 0, B)=\min\{a(E,X,B),~ \mld(Y\ni P_1,e E+B'_2),~\mld(Y\ni P_2,e E+B'_2)\}.$$
Since $\sigma:Y\rightarrow X$ is a toric morphism and $E,B'_2$ are toric divisors on $Y$, there exists a toric divisor $F_i$ over $Y$ with $\text{center}_Y F_i=P_i$ such that
$$\mld(Y\ni P_i,e E+B'_2)=a(F_i,Y,e E+B'_2)=a(F_i,X,B)$$
for $i=1,2$. This finishes the proof of the lemma.
\end{proof}
\begin{lemma}\label{lct}
Let $X=\mathbb A^2$ with coordinates $x,y$. Let $C$ be the curve $(x=0)$ and $B$ the $\mathbb R$-divisor $\lambda\cdot (x^m+y^n=0)$, where $0\leq \lambda\leq 1$  and $n,m\in \mathbb N$ such that $\lambda n\geq 1$ and $\frac{1}{n}+\frac{1}{m}\geq \lambda$.
Then
$$\lct(X\ni 0,B; C)= 1-\lambda m +\frac{m}{n}.$$
\end{lemma}
\begin{proof}
By \cite[Proposition 3.2]{HJL22}, 
$$\lct(X\ni0,0;B+tC)=\min\{\frac{n+m}{\lambda nm+tn},\frac{1}{t},\frac{1}{\lambda}\}$$
for any $t\geq 0$. The lemma follows from the fact that
$$\lct(X\ni 0,B; C)=\sup\{t\geq 0\mid \lct(X\ni0,0;B+tC)\geq 1\}.$$
\end{proof}

\begin{lemma}\label{lemexa}
Let $\epsilon=\frac{p}{q}$ be a rational number such that $\frac{1}{n+1}\leq \epsilon \leq \frac{1}{n}$, where $p,q,n\in \NN$.
Let $X=\mathbb A^2$ with coordinates $x,y$. Let $C=(x=0)$ and 
$$B=\frac{1}{qn(n+1)}\cdot \left(x^{q(2n+1)-p}+y^{2qn(n+1)}=0\right)+\frac{qn-p}{q(n+1)}\cdot(x=0).$$
Then we have $$\mult_C B\leq 1-\epsilon, \quad  \mld(X\ni 0,B)=\epsilon \quad  \text{ and }\quad \lct(X\ni 0,B; C)= \frac{(2n+1)\epsilon-1}{2n(n+1)}.$$
\end{lemma}
\begin{proof}
The assertion on $\mult_C B$ can be checked directly. The assertion on the log canonical threshold follows from Lemma \ref{lct} by direct calculation. Next we will treat the assertion on the minimal log discrepancy. By Lemma \ref{mld}, we have
\begin{align*}
\mld(X\ni 0, B)&=\inf\big\{p_1+p_2-\min\{\frac{n-\epsilon}{n+1}p_1+2p_2,\frac{n+1-\epsilon}{n}p_1\}\mid p_1,p_2\in \mathbb N\big\}\\
&=\inf\big\{\max\{\frac{\epsilon+1}{n+1} p_1-p_2,\frac{\epsilon-1}{n}p_1+p_2\}\mid p_1,p_2\in \mathbb N\big\}.
\end{align*}
Denote
$$k:=\frac{(2n+1)-\epsilon}{2n(n+1)}$$
and define a piece-wise linear function $f$ on $\mathbb R_{\geq 0}^2$ by 
$$f(r_1,r_2) :=
\begin{cases}
\frac{\epsilon+1}{n+1} r_1-r_2, & \text{if $\frac{r_2}{r_1} \leq k$,}\\
\frac{\epsilon-1}{n}r_1+r_2,  & \text{if $\frac{r_2}{r_1} \geq k$,} \\
\end{cases}
$$ 
for any $r_1,r_2\in \mathbb R_{\geq 0}$. Then we have
$$\mld(X\ni 0, B)=\inf\{f(p_1,p_2)\mid p_1,p_2\in \mathbb N\}.$$
By the definition of $f$, for any $r_1,r_2,r'_1\in \mathbb R_{\geq 0}$, we have 
$$f(r_1,r_2)\leq f(r'_1,r_2)\quad  \text{ if $r_2/r_1\leq k$ and $r'_1\geq r_1$}$$
and 
$$f(r_1,r_2)\leq f(r'_1,r_2) \quad \text{ if $r_2/r_1\geq k$ and $r'_1\leq r_1$}.$$

It is easy to check that 
$$\frac{1}{n+1}\leq k\leq \frac{1}{n} \quad \text{and}\quad \frac{2}{2n+2}\leq k\leq \frac{2}{2n+1}.$$
By direct calculation, we have $f(n,1)=f(n+1,1)=\epsilon$ and
$$f(2n+1,2)=\frac{\epsilon-1}{n}(2n+1)+2\geq \epsilon
 \quad \text{ as } \epsilon\geq \frac{1}{n+1}.$$
For any $(p_1,p_2)\in \mathbb N^2$, there are the following three cases.

(1) $p_2=1$. If $p_1\leq n$, then $f(p_1,1)\geq f(n,1)=\epsilon$; otherwise $p_1\geq n+1$, then $f(p_1,1)\geq f(n+1,1)=\epsilon$.

(2) $p_2=2$. If $p_1\leq 2n+1$, then $f(p_1,2)\geq f(2n+1,2)\geq \epsilon$; otherwise $p_1\geq 2n+2$, then $f(p_1,2)\geq f(2n+2,2)=2\epsilon$.

(3) $p_2\geq 3$. Let $p'_1=p_2/k$, then by calculation
$$f(p'_1,p_2)=\frac{(2n+1)\epsilon-1}{(2n+1)-\epsilon}\cdot p_2.$$
One can check that $f(p'_1,p_2)\geq \epsilon$ when $p_2\geq 3$ and $\epsilon\geq 1/(n+1)$. So,
either $p_1\leq p'_1$ or $p_1\geq p'_1$, we have $f(p_1,p_2)\geq f(p'_1,p_2)\geq \epsilon$.

Therefore,  
$$\mld(X\ni 0, B)=\inf\{f(p_1,p_2)\mid p_1,p_2\in \mathbb N\}=\epsilon.$$
\end{proof}

\end{document}